\documentclass[a4paper,leqno,11pt]{article}

\usepackage[utf8]{inputenc}
\usepackage[T1]{fontenc}
\usepackage[english]{babel}
\usepackage{comment}

\usepackage{geometry}

\usepackage{fancyhdr}

\usepackage{hyperref}
\hypersetup{
  colorlinks   = true, 
  urlcolor     = blue, 
  linkcolor    = blue, 
  citecolor   = black 
}

\usepackage[inline]{enumitem}

\usepackage{pbox}

\usepackage{graphicx}

\usepackage{color}

\usepackage[intlimits]{amsmath}
\usepackage{amssymb, amsfonts}

\usepackage[thmmarks, amsmath, amsthm]{ntheorem}

\usepackage{MnSymbol}
\usepackage{mathbbol}

\usepackage{bbm}

\usepackage{mathrsfs}

\usepackage[all]{xy}

\numberwithin{equation}{section}
\newtheorem{theoremcounter}{theoremcounter}[section]

\theoremstyle{plain}

\newtheorem{corollary}[theoremcounter]{Corollary}
\newtheorem{lemma}[theoremcounter]{Lemma}
\newtheorem{proposition}[theoremcounter]{Proposition}

\newtheorem{theorem}[theoremcounter]{Theorem}

\theoremnumbering{Alph}
\newtheorem{introtheorem}{Theorem}

\theoremstyle{definition}

\theoremstyle{remark}

\newtheorem{example}[theoremcounter]{Example}

\newtheorem{notation}[theoremcounter]{Notation}

\newtheorem{remark}[theoremcounter]{Remark}

\usepackage{xargs}                      
\usepackage[pdftex,dvipsnames]{xcolor}  
\usepackage[colorinlistoftodos,prependcaption,textsize=tiny]{todonotes}
\newcommandx{\unsure}[2][1=]{\todo[linecolor=red,backgroundcolor=red!25,bordercolor=red,#1]{#2}}
\newcommandx{\change}[2][1=]{\todo[linecolor=blue,backgroundcolor=blue!25,bordercolor=blue,#1]{#2}}
\newcommandx{\info}[2][1=]{\todo[linecolor=OliveGreen,backgroundcolor=OliveGreen!25,bordercolor=OliveGreen,#1]{#2}}
\newcommandx{\improvement}[2][1=]{\todo[linecolor=Plum,backgroundcolor=Plum!25,bordercolor=Plum,#1]{#2}}

\usepackage{tikz}
\usetikzlibrary{matrix,arrows}

\newcommand{\cC}{\ensuremath{\mathcal{C}}}

\newcommand{\cS}{\ensuremath{\mathcal{S}}}

\newcommand{\bZ}{\ensuremath{\mathbb{Z}}}

\newcommand{\rC}{\ensuremath{\mathrm{C}}}

\newcommand{\rE}{\ensuremath{\mathrm{E}}}

\newcommand{\rK}{\ensuremath{\mathrm{K}}}

\newcommand{\rS}{\ensuremath{\mathrm{S}}}
\newcommand{\rT}{\ensuremath{\mathrm{T}}}

\newcommand{\rV}{\ensuremath{\mathrm{V}}}

\newcommand{\rmt}{\ensuremath{\mathrm{t}}}

\newcommand{\vphi}{\ensuremath{\varphi}}

\newcommand{\eqstop}{\ensuremath{\, \text{.}}}
\newcommand{\eqcomma}{\ensuremath{\, \text{,}}}

\newcommand{\NN}{\ensuremath{\mathbb{N}}}
\newcommand{\ZZ}{\ensuremath{\mathbb{Z}}}
\newcommand{\QQ}{\ensuremath{\mathbb{Q}}}
\newcommand{\RR}{\ensuremath{\mathbb{R}}}
\newcommand{\CC}{\ensuremath{\mathbb{C}}}

\newcommand{\Hom}{\ensuremath{\mathop{\mathrm{Hom}}}}

\newcommand{\lra}{\ensuremath{\longrightarrow}}
\newcommand{\hra}{\ensuremath{\hookrightarrow}}

\newcommand{\ot}{\ensuremath{\otimes}}

\newcommand{\Cstar}{\ensuremath{\mathrm{C}^*}}

\newcommand{\Cstarred}{\ensuremath{\Cstar_\mathrm{red}}}

\newcommand{\conto}{\ensuremath{\mathrm{C}_0}}

\newcommand{\ltwo}{\ensuremath{\ell^2}}

\newcommand{\qm}{\mathbf{q}}

\newcommand{\Z}{\mathbb{Z}}

\newcommand{\Star}{\ensuremath{\mathrm{Star}}}
\newcommand{\Link}{\ensuremath{\mathrm{Link}}}

\newcommand{\cone}{\ensuremath{\mathrm{cone}}}

\newcommand{\Ext}{\ensuremath{\mathrm{Ext}}}
\newcommand{\KK}{\ensuremath{\mathrm{KK}}}

\newcommand{\Ell}{\ensuremath{\mathrm{Ell_{unord}}}}

\newcommand{\authors}{Sven Raum and Adam Skalski}
\renewcommand{\title}{Classifying right-angled Hecke C*-algebras via K-theoretic invariants}
\newcommand{\shorttitle}{Classifying Hecke C*-algebras via K-theory}

\begin{document}


\thispagestyle{empty}

\noindent
\begin{minipage}{\linewidth}
  \begin{center}
    \textbf{\LARGE \title} \\
    \authors    
  \end{center}
\end{minipage}

  

\renewcommand{\thefootnote}{}

\vspace{2em}
\noindent
\begin{minipage}{\linewidth}
  \textbf{Abstract}.
  Exploiting the graph product structure and results concerning amalgamated free products of \Cstar-algebras we provide an explicit computation of the K-theoretic invariants of right-angled Hecke \Cstar-algebras, including concrete algebraic representants of a basis in \rK-theory.  
  On the way, we show that these Hecke algebras are KK-equivalent with their undeformed counterparts and satisfy the UCT.  Our results are applied to study the isomorphism problem for Hecke \Cstar-algebras, highlighting the limits of K-theoretic classification, both for varying Coxeter type as well as for fixed Coxeter type.
\end{minipage}


\section{Introduction}
\label{sec:introduction}

The notion of Coxeter groups originates in the study of reflection groups and Weyl groups.  One of the most important tools when understanding combinatorial and representation theoretic questions associated with these groups are Hecke algebras as studied for example by Kazhdan and Lusztig \cite{kazhdanlusztig79}.  While traditionally focused on affine Coxeter types, by the time, important results on Hecke algebras could also be obtained for arbitrary Coxeter types \cite{eliaswilliamson14}.   Operator algebraic versions of Hecke algebras have been for the first time studied in \cite{davisdymarajanusykiewiczokun07}.  These have recently experienced a sharp rise in attention.  Originally introduced as a mean to study cohomology of buildings \cite[Section 11]{davis08}, Garncarek established them as an interesting class of operator algebras in their own right, by proving a clear parameter dependence of factoriality for single parameter Hecke von Neumann algebras of right-angled type.  Since then, Hecke operator algebras have been studied from various angles, informed by typical questions asked about group operator algebras.  This comes as no surprise, since first of all Hecke operator algebras can be considered deformations of the reduced group \Cstar-algebra and the group von Neumann algebra, and second for certain rational deformation parameters, Hecke algebras arise as corners of group algebras associated with groups acting on buildings, see for example \cite[Section 2]{raumskalski20}.  On the other hand a constant source of inspiration for operator algebraists is the comparison of results known for reduced group \Cstar-algebras of free groups with possible results on Hecke \Cstar-algebras.

In \cite{capsersskalskiwasilewski19}, Caspers, Skalski and Wasilewski studied MASAs of Hecke von Neumann algebras of free product type in analogy with questions about the free group factors \cite{bocaradulescu92}.  In \cite{caspers20}, Caspers obtained results on the absence of Cartan subalgebras and strong solidity for Hecke von Neumann algebras of right-angled type, which is analogue to key questions about the structure of group von Neumann algebras addressed by Ozawa \cite{ozawa04-solid} and Chifan and Sinclair \cite{chifansinclair13}.  The structure of Hecke von Neumann algebras also motivated work of Klisse \cite{klisse2020}, which offered and exploited a new perspective on certain boundaries of Coxeter complexes, earlier introduced in \cite{capracelecureux11} and \cite{lamthomas15}.

Hecke \Cstar-algebras have been put in the spotlight by the work of Caspers, Klisse and Larsen \cite{caspersklisselarsen19}, which obtains a \Cstar-simplicity result for right-angled Hecke algebras with certain deformation parameter (a complete characterization of the relevant parameter range was recently obtained in \cite{klisse2021}).   The investigation of \Cstar-simplicity of group \Cstar-algebras has developed a long tradition since Powers showed that reduced group \Cstar-algebras of free groups are simple \cite{powers75}. Recently it has experienced a breakthrough thanks to the work of Kalantar and Kennedy \cite{kalantarkennedy14-boundaries} and Breuillard, Kalantar, Kennedy and Ozawa \cite{breuillardkalantarkennedyozawa14}.  See also \cite{raum19-bourbaki} for a survey of more recent developments of the field.

It is no surprise that in the study of Hecke operator algebras, the right-angled case has received the most attention.  This is because a natural description as graph products directly relates them to amalgamated free products via a d{\'e}vissage procedure.  Free products and amalgamated free products are of major importance in the world of operator algebras since the advent of Voiculescu's free probability theory \cite{voiculescu85}.  Operator algebraic graph products have recently been introduced by Caspers and Fima \cite{caspersfima17}, laying the basis for applications of work on amalgamated free products to right-angled Hecke operator algebras.  Moreover, the particular case of Hecke operator algebras of free product type lies directly in the scope of work of Dykema \cite{dykema93-hyperfinite, dykema99} and Dykema and R{\o}rdam \cite{dykemarordam1998} on the structure and K-theory of free product \Cstar-algebras and von Neumann algebras, guiding the way to potential future results on Hecke operator algebras of general right-angled type and beyond.

The present article investigates the classification of right-angled Hecke \Cstar-algebras using K-theoretic tools.  K-theory of free product type algebras has been an object of interest for a long time.  In \cite{pimsnervoiculescu82}, Pimsner and Voiculescu provided the first calculation of $\rK$-theory of the reduced group \Cstar-algebras of free groups, and could infer absence of non-trivial projections as well as non-isomorphism for different finite numbers of generators.  Shortly after, Cuntz established a foundational perspective on $\rK$-theory of free product \Cstar-algebras in general, introducing the notion of $\rK$-amenability \cite{cuntz83-k-amenability} together with a conceptually simple calculation of the $\rK$-theory of universal free products \cite{cuntz82-internal-structure, cuntz82-free-products}.  Germain was first to establish unconditional results on the K-theory of free product \Cstar-algebras in \cite{germain96,germain97}.  An extension of this work to amalgamated free products was only recently obtained by Fima and Germain in \cite{fimagermain15}. 

The K-theory of Hecke algebras has been mainly considered in the affine case so far, where Solleveld proved a longstanding conjecture of Higson and Plymen that describes parameter independence of \mbox{K-theory} \cite{solleveld18}.  For deformed Hecke \Cstar-algebras of non-affine type, no results are stated in the literature, although the special case of Coxeter groups of free product type is covered by known results on K-theory for free product \Cstar-algebras. In the undeformed case the K-theory groups in the even case (including the right-angled one) have been computed by S\'anchez-Garcia  in \cite{sanchezgarcia07}. This was achieved  via computing the Bredon K-homology of the underlying Coxeter group via the Davis complex, and then appealing to the Baum-Connes conjecture, valid for the class of groups in question via a combination of the main results of \cite{bozejkojanuszkiewiczspatzier88}  and \cite{higsonkasparov97}. It should be noted that as often when appealing to the Baum-Connes isomorphism the arguments of \cite{sanchezgarcia07} do not lead to the explicit identification of generators of K-theory.

The main results of this paper provide an explicit computation of K-theory in both deformed and undeformed case, including concrete algebraic representants of K-theory classes, described in Notation \ref{not:clique-projections}.  These are projections $p_C \in \Cstar_q(W)$ associated with a clique $C$ in the commutation graph of a right-angled Coxeter group $W$.
\begin{introtheorem}[Corollary {\ref{cor:k-theory-hecke-algebras}}]
  \label{thm:K-theoryINTRO}
  Let $(W,S)$ be a right-angled Coxeter system.  Denote by $\Gamma$ the associated commutation graph and by $\cC$ the set of all, possibly empty, cliques of $\Gamma$.  Then the following statements hold true.
  \begin{itemize}
  \item The map $C \mapsto [p_C]$ induces an isomorphism $\ZZ^\cC \cong \rK_0(\Cstar_{\qm}(W))$.
  \item $\rK_1(\Cstar_{\qm}(W)) = 0$.
  \end{itemize}  
  In particular, $\rK_0(\Cstar_q(W))$ is a free abelian group whose rank is the number of cliques in the commutation graph of $W$.  It is generated by the projections from the copies $\Cstar(\ZZ/2\ZZ) \subset \Cstar(W)$ associated with the Coxeter generators $s \in S$.
\end{introtheorem}

Our proof of this result reduces the problem of K-theory computations to the full group \Cstar-algebra of the Coxeter group in question, by means of a KK-equivalence, analogue to K-amenability results.  Our key ingredients to establish this KK-equivalence are Fima's and Germain's work on amalgamated free products \cite{fimagermain15}, combined with the d{\'e}vissage procedure for graph product \Cstar-algebras obtained by Caspers and Fima \cite{caspersfima17}.  The K-theory computations for full group \Cstar-algebras of right-angled Coxeter groups then make use of six-term exact sequences established in \cite{fimagermain15}; this in particular gives an alternative approach to \cite[Corollary 7.3]{sanchezgarcia07}. 

Classifying \Cstar-algebras by their K-theory invariants has been one of central topics of the theory of operator algebras for almost fifty years. To wit, Elliott's classification programme for nuclear \mbox{\Cstar-algebras} \cite{elliott76} saw recent structural breakthroughs,  see e.g. \cite{gonglinniu15,tikuisiswhitewinter17,cgstw19}.  But even in the early days \mbox{K-theory} proved to be a useful invariant beyond the nuclear world too.  In particular the aforementioned result of Pimsner and Voiculescu \cite{pimsnervoiculescu82}  distinguished the reduced \Cstar-algebras of different free groups by computing their $\rK$-theory.  Generally speaking, non-isomorphism problems regarding group \mbox{\Cstar-algebras}, even in the amenable realm, are highly non-trivial, and only recent years brought the first rigidity results regarding algebras of this type; see \cite{eckhardtraum18} and references therein. Note that the group \Cstar-algebra of a right-angled Coxeter group depends only on the associated Coxeter graph, and the fact that the correspondence between (isomorphism classes of) Coxeter graphs and (isomorphism classes of) Coxeter groups is bijective in the right-angled case was established only recently (see \cite{radcliffe03} and references therein), whereas for general Coxeter groups the bijectivity is known to fail. 

Theorem \ref{thm:K-theoryINTRO} implies that there is no hope to use K-theory groups in themselves to distinguish Hecke \Cstar-algebras associated with different right-angled Coxeter groups, as is shown in Example~\ref{ex:identical-k-theory}.  Adding further data, the explicit nature of our K-theory calculations allow us to identify the trace pairing in K-theory in Proposition~\ref{prop:trace-pairing}, but also this additional information does not suffices to identify the Coxeter type, as Example~\ref{ex:identical-trace-pairing} shows.  We point out that classification results for formally similar structures such as \Cstar-algebras of right-angled Artin monoids achieved by Eilers, Li and Ruiz \cite{eilersliruiz16} are based on amenability of groupoids naturally arising when studying semigroup operator algebras.  Similar groupoid models cannot be expected for operator algebras arising from right-angled Coxeter or Artin groups.

Fixing a Coxeter type one might hope that K-theoretic information, and in particular the trace pairing suffices to recover the deformation parameter of a Hecke \Cstar-algebra.  The discussion of \mbox{K-theoretic} invariants for solving this classification problem is the second main contribution of this article.  For the purposes of our paper it is convenient to consider the following version of the Elliott invariant of a unital \Cstar-algebra $A$.

\begin{gather*}
  \label{def:Ell}
  \Ell(A) = (\rK_0(A), \rK_1(A),  [1]_{\rK_0(A)}, \rT(A), \rho_A)
  \eqcomma
\end{gather*}
where $\rT(A)$ denotes the trace simplex of $A$ and $\rho_A$ is the trace pairing of $\rK_0(A)$ and $\rT(A)$.  Combining our K-theory calculations with Dykema's work on the structure of free product \Cstar-algebras, which classified in particular their traces, allows us to determine this unordered Elliott invariant for Hecke \Cstar-algebras of free product type.  Focusing on the single parameter case, we obtain the following result, which illustrates both the power and limitations of K-theoretic invariants when it comes to classifying the algebras in our class.
\begin{introtheorem}[Theorem {\ref{thm:classification}}]
  \label{thm:classification-intro}
  Suppose $n \in \NN_{\geq 3}$ and let $q_1,q_2 \in (0,1]$.  Then the following statements hold true.
  \begin{itemize}
  \item If $\Cstar_{q_1}(\ZZ/2\ZZ^{*n}) \cong \Cstar_{q_2}(\ZZ/2\ZZ^{*n})$ and one of $q_1$ or $q_2$ lies in either of the sets $(0,\frac{1}{n-1})$, $\{\frac{1}{n-1}\}$, $(\frac{1}{n-1}\, 1]$, then so does the other.
  \item Assume $\Cstar_{q_1}(\ZZ/2\ZZ^{*n}) \cong \Cstar_{q_2}(\ZZ/2\ZZ^{*n})$ and $q_1, q_2 \in (\frac{1}{n-1}, 1]$.
    \begin{itemize}
    \item If $q_1$ or $q_2$ is irrational, then $q_1 = q_2$.
    \item If $q_1$ or $q_2$ is rational, then so is the other. Further, $\frac{1}{1 + q_1}$ and $\frac{1}{1 + q_2}$ have the same order in $\QQ/\ZZ$.
    \item If $q_1, q_2$ are rational and $\frac{1}{1 + q_1}$, $\frac{1}{1 + q_2}$ have the same order in $\QQ/\ZZ$, then the unordered Elliott invariants of $\Cstar_{q_1}(\ZZ/2\ZZ^{*n})$ and $\Cstar_{q_2}(\ZZ/2\ZZ^{*n})$ are isomorphic.
    \end{itemize}
  \item If $q_1, q_2 \in (0,\frac{1}{n-1})$, then the unordered Elliott invariants of $\Cstar_{q_1}(\ZZ/2\ZZ^{*n})$ and $\Cstar_{q_2}(\ZZ/2\ZZ^{*n})$ are isomorphic.
  \end{itemize}
\end{introtheorem}

It is natural to wonder why the version of the Elliott invariant used above does not involve the positive part of $\rK_0$.  One can expect that the order in $\rK_0$ is determined by the pairing with traces.  This indeed is known to hold in the undeformed case of $\Cstarred((\ZZ/2\ZZ)^{* n})$,  $n \geq 3$, as follows from the main result of \cite{dykemarordam1998}.  That theorem, which as far as we know is currently the best available statement regarding the relationship between the order of $\rK_0$ and traces for free products, assumes the so-called \emph{Avitzour conditions}.  These are known to be satisfied only in the undeformed case and we have not been able to determine the order structure of $\rK_0$ in the general case.  The calculation of ordered K-theory for free products not satisfying the Avitzour conditions remains an interesting open problem.

Finally we should note that it is natural to expect that the deformed algebras mentioned in Theorem~\ref{thm:classification-intro} are non-isomorphic whenever $q_1 \neq q_2$.  The opposite statement would solve the free group factor conjecture by showing that all the free group factors must be isomorphic.

\subsection*{Organisation of the article}

After this introduction, in Section~\ref{sec:preliminaries}, we recall basic facts concerning the graph products of groups and algebras, right-angled Coxeter groups, Hecke algebras and K-theory of amalgamated free products. 
In Section~\ref{sec:reduction-kk-equivalence} we establish the KK-equivalence result which implies that the K-theory groups of Hecke \Cstar-algebras do not depend on the deformation parameter. Section~\ref{sec:k-theory-computation} contains our first main result, Theorem \ref{thm:K-theoryINTRO}.  We compute the K-theoretic invariants of the Hecke-\Cstar-algebras of right-angled type, and discuss limitations of K-theoretic recovery of the Coxeter type.  In Section \ref{sec:classification}, we study the scope and limitations of K-theoretic classification of Hecke \Cstar-algebras of free product type in detail and prove Theorem \ref{thm:classification-intro}.

\subsection*{Acknowledgements}

S.R.\ was supported by the Swedish Research Council through grant number 2018-04243 and by the European Research Council (ERC) under the European Union's Horizon 2020 research and innovation programme (grant agreement no. 677120-INDEX).
A.S.\ was partially supported by the National Science Center (NCN) grant no.~2014/14/E/ST1/00525. 

The major part of our work on this article was completed while S.R. was visiting IM PAN in autumn 2020.  We would like to thank Piotr Nowak for making this visit possible.  We are also grateful to him for helpful conversations on the structure and K-theory of Hecke \Cstar-algebras. We thank Julian Kranz for clarifying comments regarding the UCT results and Klaus Thomsen and Jamie Gabe for pointing out an imprecision in an earlier version of this article. Finally we are very grateful to  the referee for several useful comments, and in particular for pointing to us the reference \cite{sanchezgarcia07}.

\section{Preliminaries}
\label{sec:preliminaries}

In this section we recall basic facts concerning the graph products of groups and algebras, right-angled Coxeter groups, Hecke algebras and K-theory of amalgamated free products.  We also show how in the case of algebras with free K-theory one can deduce from six-term exact sequences stability of the UCT under the amalgamated free products. 

\begin{notation}
  Multi-indices will be denoted by a bold font throughout this article. We assume throughout that the \Cstar-algebras we consider are separable.
\end{notation}

\subsection{K-theory and the UCT for amalgamated free products}
\label{sec:k-theory-amalgamated-free-products}

For information regarding $\rK$-theory and $\KK$-theory we refer to \cite{blackadar98}.  A detailed introduction to reduced and universal amalgamated free products can be found in \cite[Section 2]{fimagermain15}.

The identification of the $\rK$-theory groups of amalgamated free products has been achieved by Fima and Germain in \cite{fimagermain15}.  Concretely speaking, \cite[Corollary 4.12]{fimagermain15} implies that if $A_1, B, A_2$ are unital separable \Cstar-algebras, with unital inclusions $\alpha_1: B \to A_1$ and $\alpha_2: B \to A_2$, then for every separable \Cstar-algebra $C$ there are the following six-term short exact sequences in $\rK$-theory and $\KK$-theory:
\begin{gather*}
\xymatrix{
	\rK_0(B) \ar[r] & \rK_0(A_1) \oplus \rK_0(A_2) \ar[r] & \rK_0(A_1 *_B A_2) \ar[d] \\
	\rK_1(A_1 *_B A_2) \ar[u] & \rK_1(A_1) \oplus \rK_1(A_2) \ar[l] & \rK_1(B) \ar[l]
}
\end{gather*}
\begin{gather*}
\xymatrix{
	\KK^0(B,C) \ar[d] & \KK^0(A_1, C) \oplus \KK^0(A_2, C) \ar[l] & \KK^0(A_1 *_B A_2, C) \ar[l] \\
	\KK^1(A_1 *_B A_2, C) \ar[r] & \KK^1(A_1, C) \oplus \KK^1(A_2, C) \ar[r] & \KK^1(B, C) \ar[u]
}
\end{gather*}

Recall from \cite[Chapter 23]{blackadar98} that a separable \Cstar-algebra $A$ satisfies the UCT (the Universal Coefficient Theorem) if the natural short exact sequence 
\begin{gather*}
0 \to \Ext_\ZZ^1( \rK_*(A), \rK_*(C)) \lra \KK^*(A,C) \stackrel{\gamma}{\lra} \Hom(\rK_*(A), \rK_*(C)) \to 0
\end{gather*}
is exact for every separable \Cstar-algebra $C$.  Here the map $\gamma$ is the dual of the $\KK$-pairing $\rK_*(A) \otimes \KK^*(A,C) \to \rK_*(C)$, and as such is natural in $A$ and $C$. We shall now record a consequence of the results of \cite{fimagermain15}. 

\begin{proposition}
  \label{thm:UCT-am}
  Let $A_1, A_2, B$ be  unital separable \Cstar-algebras with unital inclusions with conditional expectations $\alpha_1: B \to A_1$ and $\alpha_2: B \to A_2$.  If $A_1, A_2, B$ satisfy the UCT,
   then $A_1 *_B A_2$ also satisfies the UCT.
\end{proposition}
\begin{proof}
Under the assumptions above \cite[Theorem 4.1]{fimagermain15} says that the suspension of $A_1 *_B A_2$ is KK-equivalent to the mapping cone determined by the direct sum of the inclusion maps $B \subset A_1$ and $B \subset A_2$ (see the proof of Proposition \ref{prob:amalgamated-free-products-preserve-kk-equivalence} for more details). As mapping cones of the algebras satisfying the UCT also satisfy the UCT, and the property of satisfying the UCT is stable under taking direct products and also under passing to suspensions, the result follows.

\end{proof}

If in the context of the above proposition one assumes in addition that the $\rK_*$-groups of $A_1$, $A_2$, $B$ and $A_1 *_B A_2$ are free (as will be the case in our application), one can give a direct proof, using only the six-term exact sequences.

\subsection{Graph products of groups and algebras}
\label{sec:graph-products}

Throughout the paper, by a graph $\Gamma=(\rV, \rE)$ we will understand a finite, and possibly empty, simplicial graph.  Formally, $\rV$ denotes a finite set and $\rE \subset \rV \times \rV$ is a symmetric subset that does not intersect  the diagonal.  We will sometimes write $\rV(\Gamma)$ for $V$ and $\rE(\Gamma)$ for $E$.
We denote the \emph{induced subgraph} on a subset $V' \subset \rV(\Gamma)$ by $\Gamma|_{V'}$, so that $\rE(\Gamma|_{V'}):= \rE \cap (\rV' \times \rV')$.  Given a vertex $v \in \rV$, we denote by $\Gamma \setminus v$ the induced subgraph of $\Gamma$ with vertex set $\rV \setminus \{v\}$.  A \emph{clique} in a graph $\Gamma$ is a complete subgraph, i.e.\ a subgraph $\Gamma'$ of $\Gamma$ such that $\rE(\Gamma') = \Gamma' \times \Gamma' \setminus\{(\gamma, \gamma):\gamma \in \Gamma'\}$.  For an element $w \in \rV$ we consider the graphs $\Link(w)= \Gamma|_{\{v \in \rV:(v,w) \in \rE \}}$ and $\Star(w) = \Gamma|_{\{v \in \rV:(v,w) \in \rE \} \cup \{w\}}$.

We will need in Section \ref{sec:k-theory-computation} a simple lemma counting cliques in a given graph.  Recall that the empty set is a clique in every graph.
\begin{lemma}
  \label{lem:counting-cliques}
  Denote by $N(\Gamma)$ the number of cliques of arbitrary size in a graph $\Gamma$.  Then for all $v \in \rV(\Gamma)$
  \begin{gather*}
    N(\Gamma) = N(\Link(v)) + N(\Gamma \setminus v)
    \eqstop
  \end{gather*}
\end{lemma}
\begin{proof}
  For a clique $C$ in $\Gamma$, there are precisely two options.  Either $v \in \rV(C)$ and $C$ is contained in $\Star(v)$ or $v \notin \rV(C)$ and $C$ is contained in $\Gamma \setminus v$.  In the former case, $C \setminus v$ is a clique in $\Link(v)$, and every clique in $\Link(v)$ is of this form.
\end{proof}

Suppose that $\Gamma=(\rV,\rE)$ is a graph and that for every $v \in \rV$ we have an associated group $G_v$. Based on this data, Green introduced in \cite{green-thesis} a construction of the graph product of groups, leading to a group $G_\Gamma$, generated by copies of all $G_v$, $v \in \rV$, such that $G_v$ commutes with $G_w$ if and only if $(v,w) \in \rE$.  In \cite{caspersfima17} Caspers and Fima transferred this construction to the operator algebra world, taking into account various possible variants.  For us the key notions will be the \emph{universal} graph product of unital \Cstar-algebras, and the \emph{reduced} graph product of unital \Cstar-algebras equipped with faithful states.  Thus if for every $v \in \rV$ we are given a unital \Cstar-algebra $A_v$ (with a faithful state $\vphi_v$), we have the corresponding universal graph product  $\prod_{\Gamma} A_v$ (and the reduced graph product  $\prod_{\Gamma} (A_v, \vphi_v)$). The constructions are closely related to the amalgamated free product of \Cstar-algebras, and in particular \cite[Section 2]{caspersfima17} contains the following identifications, inspired by Serre's d\'evissage procedure for groups acting on trees. 
\begin{theorem}[See Theorem 2.15 and Remark 2.16 of {\cite{caspersfima17}}]
  \label{thm:devissage-graph-products}
  Let $\Gamma = (\rV,E)$ be a  graph and $(A_v)_{v \in \rV}$ a family of unital \Cstar-algebras.  Let $w \in \rV$.
  \begin{itemize}
  \item There is an isomorphism of \Cstar-algebras $\prod_{\Gamma} A_v \to \left (\prod_{\Star(w)} A_v \right ) *_{\prod_{\Link(w)} A_v} \left ( \prod_{\Gamma \setminus w} A_v \right )$, commuting with the natural inclusions of $(A_v)_{v \in \rV}$ in each side. 
  \item If $\vphi_v$ is a faithful state on $A_v$ for all $v \in \rV$, then there is an isomorphism of \Cstar-algebras $\prod_{\Gamma} (A_v, \vphi_v) \to \left (\prod_{\Star(w)} (A_v,\vphi_v) \right ) *_{\prod_{\Link(w)} (A_v, \vphi_v)} \left ( \prod_{\Gamma \setminus w} (A_v, \vphi_v) \right )$, commuting with the natural inclusions of $(A_v)_{v \in \rV}$ in each side. 
  \end{itemize}
\end{theorem}

\subsection{Coxeter systems and Hecke algebras}
\label{sec:coxeter-systems}
A good introduction to Coxeter groups can be found for example in \cite{humphreys90}. Here we will consider only the finitely generated right-angled case.  A \emph{Coxeter matrix of right-angled type} over a non-empty finite set $S$ is a symmetric $S \times S$ matrix whose diagonal entries equal $1$  and whose off-diagonal entries lie in the set $ \{2,\infty\}$.  Let $M = (m_{st})_{s,t \in S}$ be a Coxeter matrix.  The \emph{Coxeter system} $(W, S)$ associated with $M$ is given by the group 
\begin{gather*}
W  = \langle S \mid \forall s,t \in S: \, s^2 = e \textup{ and } st = ts \textup{ whenever } m_{st}=2 \rangle
\end{gather*}
together with its generating set $S$. Given $w \in W$ we denote by $|w|$ its \emph{word length} with respect to the generating set $S$. The data contained in the Coxeter matrix can be equivalently encoded in the \emph{commutation graph} $\Gamma_{(W,S)}$, with the vertex set $S$ and the edge set $\{(s,t)\in S \times S: m_{st}=2\}$.  Note that the commutation graph is the complement of the standard (unlabelled) \emph{Coxeter graph} of $(W,S)$.  In \cite{caspersklisselarsen19} the commutation graph was called Coxeter diagram, which is unfortunately in conflict with the standard terminology employed in the theory of Coxeter groups.  It is easy to see, and was in fact the original inspiration to study graph products of groups, that $W$ is isomorphic to the graph product of groups of order $2$; the relevant isomorphism  $W \cong \prod_{\Gamma_{(W,S)}} \ZZ/2\ZZ$ preserves generating sets.

The Coxeter matrix and its Coxeter system are called \emph{irreducible} if the Coxeter matrix is not the Kronecker tensor  
product of two non-trivial Coxeter matrices.  Equivalently, there is no partition $S = S_1 \sqcup S_2$ such that $W = \langle S_1 \rangle \oplus \langle S_2 \rangle$.

Given $W$ as above  we will write $\Cstar(W)$ for the universal group \Cstar-algebra of $W$ and $\Cstar_r(W)$ for its reduced group \Cstar-algebra.  The central object of study in this paper is the family of unital \Cstar-algebras arising as deformations of $\Cstar_r(W)$, the so-called Hecke \Cstar-algebras of right-angled Coxeter groups. For the origins of the construction and its connections to geometry we refer to \cite{davis08} and \cite[Section 2.4]{raumskalski20}; many properties of the algebras in the class were recently established in \cite{caspersklisselarsen19}. Here we will take a direct way of introducing the Hecke operator algebras as generated by certain unitary representations of $W$.

Fix $(W,S)$ as above and a multiparameter $\qm \in (0,1]^S$. Consider $\lambda_{\qm}: W \to B(\ltwo(W))$, the unitary representation given by the formulas ($s\in S$, $w \in W$)
\begin{gather} \label{def_lambdaq}
  \lambda_\qm(s)\delta_w =
  \begin{cases}
    \phantom{-} \frac{1-q_s}{1+q_s} \delta_w + \frac{2 q^{1/2}}{1 + q_s} \delta_{sw} & \text{if } |sw| > |w|, \\
    -\frac{1-q_s}{1+q_s} \delta_w + \frac{2 q^{1/2}}{1 + q_s} \delta_{sw} & \text{if } |sw| < |w| \eqstop
  \end{cases}
\end{gather}
We then denote by  $\Cstar_\qm(W)$ the \Cstar-algebra generated by $\lambda_\qm(W)$. Note that for $\qm={\bf{1}}$ the representation $\lambda_\qm$ is the left regular representation, so that we have  $\Cstar_\textbf{1}(W)= \Cstar_r(W)$. The representations $\lambda_\qm$ can be also viewed as graph product representations, which was the point of view exploited in \cite{caspers20, caspersklisselarsen19} and systematically investigated in \cite[Section 4]{raumskalski20}. Each of them is equipped with a canonical trace $\tau:\Cstar_\qm(W) \to \CC$, given as the vector state associated with $\delta_e \in \ell^2(W)$. Finally if $q \in (0,1]$ we will write simply $\Cstar_q(W)$ for $\Cstar_\qm(W)$ with $\qm = (q,\ldots,q)\in (0,1]^S$.

The group isomorphism mentioned earlier and the general properties of the graph products of groups and operator algebras, \cite[Example 2.6]{caspersfima17} yield the natural isomorphism
\begin{gather*}
\Cstar(W) \cong \prod_{\Gamma_{(W,S)}} \Cstar(\ZZ/2\ZZ).
\end{gather*}
Further by \cite[Corollary 3.4]{caspers20}, see also \cite[Section 1.10]{caspersklisselarsen19}, if we view $\Cstar(\ZZ/2\ZZ) \cong \CC \chi_1 \oplus \CC \chi_{-1}$ via the Fourier transform
and  for each $s \in S$ we denote by $\vphi_s \in \cS(\Cstar(\ZZ/2\ZZ))$ the state given by
\begin{align*}
\vphi_s(\chi_1) & = \frac{1}{2} \langle (1 + \lambda_\qm(s)) \delta_e, \delta_e \rangle = \frac{1}{1 + q_s},\\
\vphi_s(\chi_{-1}) & = 1 - \vphi_s(\chi_1) = \frac{q_s}{1 + q_s},
\end{align*}
then we also have the trace preserving isomorphism
\begin{gather*}
\Cstar_\qm(W) \cong \prod_{\Gamma_{(W,S)}} (\Cstar(\ZZ/2\ZZ), \vphi_s),
\end{gather*}
extending the isomorphism of groups $W \to \prod_{\Gamma(W,S)} \ZZ/2\ZZ$.

\section{The reduction map is a KK-equivalence}
\label{sec:reduction-kk-equivalence}

In this section we will prove that the unitary representations $\lambda_\qm: W \to \Cstar_\qm(W)$ described in Section \ref{sec:coxeter-systems} induce a KK-equivalence of the Hecke-\Cstar-algebra with $\Cstar(W)$, in analogy with K-amenability results \cite{cuntz83-k-amenability}.  This implies that the K-theory of Hecke \Cstar-algebras associated to right-angled Coxeter groups does not depend on the deformation parameter.  In particular, it coincides with the K-theory of the reduced group \Cstar-algebra of the underlying group.  We will derive this result from a KK-equivalence between universal graph products and reduced graph products of unital \Cstar-algebras. 

Once again we refer to \cite[Section 2]{fimagermain15} for information on amalgamated free products of \Cstar-algebras.  In the next proposition we will use the approach to $\mathbf{KK}$-category via the language of triangulated categories, as developed by Meyer and Nest in \cite{meyernest06}.  We refer the reader to \cite{neeman01} for a treatment of triangulated categories.
\begin{proposition}
  \label{prob:amalgamated-free-products-preserve-kk-equivalence}
  For $i \in \{1,2\}$, let $A_i, B_i, C_i$ be unital separable \Cstar-algebras with unital inclusions with conditional expectations $\alpha_i: C_i \to A_i$ and $\beta_i: C_i \to B_i$.  If there is a commutative diagram of unital *-homomorphisms
  \begin{gather*}
    \xymatrix{
      A_1 \ar[d]^{\vphi_A} & \ar[l] C_1 \ar[r] \ar[d]^{\vphi_C}  & B_1 \ar[d]^{\vphi_B} \\
      A_2  & \ar[l] C_2 \ar[r] & B_2
    }
  \end{gather*}
  such that $\vphi_A, \vphi_B, \vphi_C$ are $\rK\rK$-equivalences, then the induced *-homomorphism $\vphi: A_1 *_{C_1} B_1 \to A_2 *_{C_2} B_2$ is a $\rK\rK$-equivalence too.
\end{proposition}
\begin{proof}
  Throughout the proof we denote by $\rS$ the suspension functor for \Cstar-algebras, realised by tensoring with $\conto((-1,1))$. Recall that given a unital $^*$-homomorphism $\alpha:A \to B$ between two unital $\Cstar$-algebras, its cone is defined as
\[  \mathrm{cone} (\alpha) =  \{(f, a) \in \conto([0,1), B) \oplus A \mid f(0) = \alpha(a) \}.\]
We shall consider the Kasparov category $\mathbf{KK}$ whose objects are \Cstar-algebras, and whose Hom-spaces are KK-groups with the Kasparov product defining the composition.  Slightly extending the category of objects and adapting the morphisms leads to a triangulated category $\boldsymbol{\widetilde{KK}}$, as described in \cite[Section 2 and Appendix A]{meyernest06}.  The only aspect to take care of when trying to view $\mathbf{KK}$ as a triangulated category, which neccessitates the extension mentioned above, is that the suspension functor on $\mathbf{KK}$ is not an isomorphism, but merely an equivalence of categories.  As in \cite{meyernest06}, we will not make the formal distinction between $\boldsymbol{\widetilde{KK}}$ and $\mathbf{KK}$.
By \cite[Section 2.1]{meyernest06} the mapping cone of a *-homomorphism $\alpha: A \to B$ gives rise to a distinguished triangle
\begin{gather*}
\rS B \to \cone(\alpha) \to A \stackrel{\psi}{\to} B
\eqstop
\end{gather*}

  Let $i \in \{1,2\}$ and consider the \Cstar-subalgebra
  \begin{gather*}
    D_i
    =
    \{f \in \rS (A_i *_{C_i} B_i) \mid f((-1,0)) \subset A_i, f((0, 1)) \subset B_i, \text{ and } f(0) \in C_i \}
    \subset
    \rS(A_i *_{C_i} B_i)
    \eqcomma
  \end{gather*}
which is *-isomorphic with the mapping cone of $\alpha_i \oplus \beta_i: C_i \to A_i \oplus B_i$, via the map that takes $(f,c) \in \mathrm{cone}(\alpha_i \oplus \beta_i)$
  to the function
  \begin{gather*}
    t \mapsto
    \begin{cases}
      \pi_A(f(-t)) & \text{ if } t < 0 \\
      c & \text{ if } t = 0 \\
      \pi_B(f(t)) & \text{ if } t > 0
      \eqstop
    \end{cases}
  \end{gather*}
  Then \cite[Theorem 4.1]{fimagermain15} says that the inclusion $D_i \hra \rS(A_i *_{C_i} B_i)$ is a KK-equivalence.

 Applying the construction mentioned in the beginning of the proof to the *-homomorphism $\alpha_i \oplus \beta_i: C_i \to A_i \oplus B_i$, we obtain a distinguished triangle
  \begin{gather*}
    \rS(A_i \oplus B_i) \to \cone(\alpha_i \oplus \beta_i) \to C_i \to A_i \oplus B_i
    \eqstop
  \end{gather*}
  By the axiom TR2 of a triangulated category (see e.g. \cite[p.29f]{neeman01}) also the shifted triangle
  \begin{gather*}
    \rS C_i \to \rS(A_i \oplus B_i) \to \cone(\alpha_i \oplus \beta_i) \to C_i
  \end{gather*}
  is distinguished.  Denoting by $\vphi_{\mathrm{cone}}:\cone(\alpha_1 \oplus \beta_1) \to \cone(\alpha_2 \oplus \beta_2)$ the *-homomorphism induced by $\vphi_A \oplus \vphi_B$, we observe that the diagram
  \begin{gather*}
    \xymatrix{
      \rS C_1 \ar[r] \ar[d]^{\rS \vphi_C} & \rS(A_1 \oplus B_1) \ar[r] \ar[d]^{\rS(\vphi_A \oplus \vphi_B)} & \cone(\alpha_1 \oplus \beta_1) \ar[r] \ar[d]^{\vphi_{\cone}} & C_1 \ar[d]^{\vphi_C} \\
      \rS C_2 \ar[r] & \rS(A_2 \oplus B_2) \ar[r] & \cone(\alpha_2 \oplus \beta_2) \ar[r] & C_2
      }
  \end{gather*}
  commutes.  Since $\rS \vphi_C$ and $\rS(\vphi_A \oplus \vphi_B)$ are KK-equivalences, \cite[Proposition 1.1.20]{neeman01} says that $\vphi_{\mathrm{cone}}$ is a KK-equivalence too.  Consider now the commuting square
  \begin{gather*}
    \xymatrix{
      \cone(\alpha_1 \oplus \beta_1) \ar[d]^{\vphi_\mathrm{cone}}  \ar[r]^{\cong_{\mathbf{KK}}} & \rS(A_1 *_{\rC_1} B_1) \ar[d]^{\rS \vphi} \\
      \cone(\alpha_2 \oplus \beta_2) \ar[r]^{\cong_{\mathbf{KK}}} & \rS(A_2 *_{\rC_2} B_2)
      \eqstop
    }
  \end{gather*}
  Since three out of four morphisms are KK-equivalences, it follows that also $\rS \vphi$ is a KK-equivalence, and so is $\vphi$.
\end{proof}

We can now combine the last result with the d{\'e}vissage process for graph product \Cstar-algebras described in Theorem \ref{thm:devissage-graph-products} in order to establish a general theorem on the KK-equivalence of reduced and universal graph products of \Cstar-algebras.  We use Skandalis' notion of $\rK$-nuclearity \cite{skandalis88}.
\begin{theorem}
  \label{thm:reduction-graph-product-kk-equivalence}
  Let $\Gamma = (\rV,E)$ be a graph and let $(A_v, \vphi_v)_{v \in \rV}$ be a family of unital, separable and K-nuclear \Cstar-algebras with faithful tracial states.  Then the reduction map $\prod_{\Gamma} A_v \to \prod_{\Gamma} (A_v, \vphi_v)$ is a KK-equivalence. If moreover each of the algebras $A_v$ satisfies the UCT, so do the graph products $\prod_{\Gamma} A_v$ and $\prod_{\Gamma} (A_v, \vphi_v)$. 
\end{theorem}
\begin{proof}
  We prove the first part of the theorem by induction on the number of vertices $|\rV|$.  When $\rV$ is a complete graph, and in particular in the trivial base case of $|V| = 1$,  the statement follows from the K-nuclearity assumption. Indeed, in this case the reduced graph product is just the minimal tensor product of \mbox{\Cstar-algebras}, and the universal graph product is the corresponding maximal tensor product. The reduction map is then a KK-equivalence by the inductive application of  \cite[Proposition 3.5 (a)]{skandalis88}. 
  
  Fix then $\Gamma = (\rV, \rE)$ and $(A_v, \vphi_v)_{v \in \rV}$ as in the statement of the theorem and assume that the theorem is established for all graphs whose number of vertices is strictly lower than that of $\rV$.  By our initial remark, we may assume that $\Gamma$ is not a complete graph, so that we can choose $w \in V$ such that $\Star(w)$ has fewer vertices than $\rV$.  By Theorem \ref{thm:devissage-graph-products}, there are natural isomorphisms
  \begin{align*}
    \prod_{\Gamma} A_v
    & \cong
      \left (\prod_{\Star(w)} A_v \right ) *_{\mathrm{max}, \prod_{\Link(w)} A_v} \left ( \prod_{\Gamma \setminus w} A_v \right ) \\
    \prod_{\Gamma} (A_v, \vphi_v)
    & \cong
      \left (\prod_{\Star(w)} (A_v,\vphi_v) \right ) *_{\mathrm{red}, \prod_{\Link(w)} (A_v, \vphi_v)} \left ( \prod_{\Gamma \setminus w} (A_v, \vphi_v) \right )
    \eqstop
  \end{align*}
  We distinguish explicitly between maximal and reduced amalgamated free product, since we will also apply the former version to \Cstar-algebras with a prescribed state.  So we obtain a description of the reduction map $\prod_{\Gamma} A_v \to \prod_{\Gamma} (A_v, \vphi_v)$ by the following diagram.
  \begin{gather*}
    \xymatrix{
      \prod_{\Gamma} A_v \ar[dd] \ar[r]^-{\cong} & \prod_{\Star(w)} A_v *_{\mathrm{max}, \prod_{\Link(w)} A_v} \prod_{\Gamma \setminus w} A_v \ar[d]^\alpha \\
      & \left (\prod_{\Star(w)} (A_v, \vphi_v) \right)  *_{\mathrm{max}, \prod_{\Link(w)} (A_v, \vphi_v)} \left( \prod_{\Gamma \setminus w} (A_v, \vphi_v) \right ) \ar[d]^\beta\\
      \prod_{\Gamma} (A_v, \vphi_v) \ar[r]^-{\cong} & \left (\prod_{\Star(w)} (A_v, \vphi_v) \right)  *_{\mathrm{red}, \prod_{\Link(w)} (A_v, \vphi_v)} \left( \prod_{\Gamma \setminus w} (A_v, \vphi_v) \right )
      }
    \end{gather*}
    We now use the fact that the traces on $A_v$ provide conditional expectations allowing us to apply Proposition~\ref{prob:amalgamated-free-products-preserve-kk-equivalence} and results from \cite{fimagermain15}. The map $\beta$ is a KK-equivalence by \cite[Corollary 4.12]{fimagermain15}.  The map $\alpha$ is a KK-equivalence by Proposition~\ref{prob:amalgamated-free-products-preserve-kk-equivalence} and the $\rK$-nuclearity assumption on $A_w$.  So also the reduction map of the graph product is a KK-equivalence.
    
    The second statement can be shown along the same lines, using  Proposition \ref{thm:UCT-am} and the fact -- observed by Skandalis in \cite[Section 5]{skandalis88} -- that minimal (as well as maximal) tensor products of \Cstar-algebras satisfying the UCT also satisfy the UCT. Note also that a \Cstar-algebra satisfying the UCT is neccessarily K-nuclear.
\end{proof}

For the formulation of the next corollary, recall the representation $\lambda_\qm$ of a right-angled Coxeter group $W$ defined in \eqref{def_lambdaq} and denote its canonical extension to a unital $^*$-homomorphism from $\Cstar(W)$ to $\Cstar_\qm(W)$ by the same symbol.
\begin{corollary}
  \label{thm:reduction-racg-kk-equivalence}
  Let $(W,S)$ be a right-angled Coxeter system and $\qm \in \RR_{>0}^S$.  Then $\lambda_\qm: \Cstar(W) \to \Cstar_\qm(W)$ is a KK-equivalence. Moreover the algebra $\Cstar_\qm(W)$ satisfies the UCT.
\end{corollary}
\begin{proof}
  The identifications introduced in Subsection \ref{sec:coxeter-systems}  imply that we obtain a commutative diagram
  \begin{gather*}
    \xymatrix{
      \Cstar(W) \ar[r]^-{\cong} \ar[d] & \prod_{\Gamma_{(W,S)}} \Cstar(\ZZ / 2\ZZ) \ar[d] \\
      \Cstar_q(W) \ar[r]^-{\cong} & \prod_{\Gamma_{(W,S)}} (\Cstar(\ZZ/2\ZZ), \vphi_s) \eqstop
      }
    \end{gather*}
    By Theorem \ref{thm:reduction-graph-product-kk-equivalence}, the right vertical morphism is a KK-equivalence.  So also the left vertical morphism is a KK-equivalence.
    
    The last statement follows from the above identification and the second part of Theorem \ref{thm:reduction-graph-product-kk-equivalence}. 
  \end{proof}

One could alternatively deduce the UCT for $\Cstar(W)$ from the results of \cite{tu99}, as by \cite{bozejkojanuszkiewiczspatzier88} all Coxeter groups have the Haagerup property.

\section{Computing K-theoretic invariants of Hecke \Cstar-algebras}
\label{sec:k-theory-computation}

In this section we first provide an alternative (to that of \cite{sanchezgarcia07}) method of computing the K-theory for the full group \Cstar-algebra of a right-angled Coxeter system $(W,S)$ , allowing us to describe explicit generators in terms of the  commutation graph of $(W,S)$.  Combined with the KK-equivalence established in the previous section, this also provides a very concrete description of the K-theory of the corresponding Hecke \Cstar-algebra.  We then exploit the explicit nature of our calculations in order to describe the trace pairing.

In order to formulate our $\rK$-theory calculations of universal group \Cstar-algebras, let us introduce the following notation.

\begin{notation}
  \label{not:clique-projections}
  Let $(W,S)$ be a right-angled Coxeter system with commutation graph $\Gamma$.  Recall from the last part of Section \ref{sec:coxeter-systems} the minimal projection $\chi_1 \in \Cstar(\ZZ/2\ZZ)$ associated with the non-trivial irreducible representation of $\ZZ/2\ZZ$.   We associate with each non-empty clique $C$ of $\Gamma$ the projection
  \begin{gather*}
    p_C = \ot_{s \in C} \chi_1^{(s)} \in \prod_\Gamma \Cstar(\ZZ/2\ZZ) \cong \Cstar(W)
    \eqcomma
  \end{gather*}
  where by $\chi_1^{(s)}$ we denote the minimal projection in the copy of $\Cstar(\ZZ/2\ZZ)$ generated by $s \in C \subset S$, and set $p_\emptyset = 1_{\Cstar(W)}$.  
\end{notation}

\begin{theorem}
  \label{thm:k-theory-computation}
  Let $(W,S)$ be a right-angled Coxeter system.  Denote by $\Gamma$ the associated commutation graph and by $\cC$ the set of all cliques of $\Gamma$ (including the empty clique).  Then the following statements about $\rK$-theory of the universal group \Cstar-algebra hold true.
  \begin{itemize}
  \item The map $C \mapsto [p_C]$ induces an isomorphism $\ZZ^\cC \cong \rK_0(\Cstar(W))$.
  \item $\rK_1(\Cstar(W)) = 0$.
  \end{itemize}
  In particular, $\rK_0(\Cstar(W))$ is a free abelian group whose rank is the numbers of cliques in the commutation graph of $(W,S)$ and it is generated by the projections from the copies $\Cstar(\ZZ/2\ZZ) \subset \Cstar(W)$ associated with the Coxeter generators $s \in S$.
\end{theorem}
\begin{proof}
  We will make use of the natural identification of $\Cstar(W)$ with the graph product $ \prod_{\Gamma} \Cstar(\ZZ/2\ZZ)$ described in the last part of Section \ref{sec:coxeter-systems}.  In particular, this identification preserves the inclusion of the copies of $\Cstar(\ZZ/2\ZZ)$ associated with the Coxeter generators in $S$.

  Let us first observe that the statement of the theorem follows from direct K-theory computations if $W$ is abelian, and in particular if $W$ has rank 1.   Let $(W, S)$ be non-abelian of rank at least 2 and assume that the result is proven for all right-angled Coxeter groups whose rank is strictly lower than the rank of $W$.

  We first show that $\rK_1(\Cstar(W)) = 0$.  Since $W$ is not abelian, there is $s \in S$ such that $\Star(s) \neq \Gamma$.  Then we have an amalgamated free product decomposition
  \begin{gather*}
    \prod_{\Gamma} \Cstar(\ZZ/2\ZZ) \to \left (\prod_{\Star(s)} \Cstar(\ZZ/2\ZZ) \right ) *_{\prod_{\Link(s)} \Cstar(\ZZ/2\ZZ)} \left ( \prod_{\Gamma \setminus s} \Cstar(\ZZ/2\ZZ) \right )
    \eqcomma
  \end{gather*}
  by Theorem \ref{thm:devissage-graph-products}.  Note that the isomorphism $\prod_{\Star(s)} \Cstar(\ZZ/2\ZZ)  \cong \Cstar(\ZZ/2\ZZ) \otimes \prod_{\Link(s)} \Cstar(\ZZ/2\ZZ)$ preserves the natural inclusions of $\prod_{\Link(s)} \Cstar(\ZZ/2\ZZ)$.  This implies that the inclusion $\prod_{\Link(s)} \Cstar(\ZZ/2\ZZ) \to \prod_{\Star(s)} \Cstar(\ZZ/2\ZZ)$ has a right inverse and thus induces an injection in K-theory.  Using the 6-term exact sequence for amalgamated free products from \cite[Corollary 4.12]{fimagermain15}, which was recalled in Section~\ref{sec:k-theory-amalgamated-free-products}, we obtain the following 6-term exact sequence.
  \begin{gather*}
    \xymatrix{
      \rK_0(\prod_{\Link(s)} \Cstar(\ZZ/2\ZZ)) \ar[r] &
      \pbox{12em}{
        $\rK_0(\prod_{\Star(s)} \Cstar(\ZZ/2\ZZ)) \oplus$ \\ $\phantom{m} \rK_0(\prod_{\Gamma \setminus s} \Cstar(\ZZ/2\ZZ))$
      }
      \ar[r]  &
      \rK_0(\prod_{\Gamma} \Cstar(\ZZ/2\ZZ)) \ar[d] \\
      \rK_1(\prod_{\Gamma} \Cstar(\ZZ/2\ZZ)) \ar[u] &
      \pbox{12em}{
        $\rK_1(\prod_{\Star(s)} \Cstar(\ZZ/2\ZZ)) \oplus$ \\ $\phantom{m} \rK_1(\prod_{\Gamma \setminus s} \Cstar(\ZZ/2\ZZ))$
      }
      \ar[l] &
      \rK_1(\prod_{\Link(s)} \Cstar(\ZZ/2\ZZ)) \ar[l]
    }
  \end{gather*}
  Applying the induction hypothesis, we obtain the exact sequence
  \begin{gather*}
    \xymatrix{
      0 \ar[r] &
      \rK_1(\prod_{\Gamma} \Cstar(\ZZ/2\ZZ)) \ar[r] &
      \rK_0(\prod_{\Link(s)} \Cstar(\ZZ/2\ZZ))  \ar@{->} `r[d] `[l] `^d[ll] `^r[dll] [d]\\
      && \rK_0(\prod_{\Star(s)} \Cstar(\ZZ/2\ZZ) ) \oplus \rK_0(\prod_{\Gamma \setminus s} \Cstar(\ZZ/2\ZZ))
    }
  \end{gather*}
  Since the map $\rK_0(\prod_{\Link(s)} \Cstar(\ZZ/2\ZZ)) \to \rK_0(\prod_{\Star(s)} \Cstar(\ZZ/2\ZZ) )$ is injective, it follows that $\rK_1(\Cstar(W)) \cong \rK_1(\prod_{\Gamma} \Cstar(\ZZ/2\ZZ)) = 0$.

  To compute $\rK_0(\Cstar(W))$ we first notice that the short exact sequence
  \begin{gather*}
    \xymatrix{
      0 \ar[r] & \rK_0(\prod_{\Link(s)} \CC^2) \ar[r] & \rK_0(\CC^2 \ot \prod_{\Link(s)} \CC^2 ) \oplus \rK_0(\prod_{\Gamma \setminus s} \CC^2) \ar[r] & \rK_0(\prod_\Gamma \CC^2) \ar[r] & 0
    }
  \end{gather*}
  combined with the induction hypothesis shows that $\rK_0(\Cstar(W)) \cong \rK_0(\prod_\Gamma \CC^2)$ is free abelian and generated by the projections coming from the inclusions $\Cstar(\langle s \rangle) \subset \Cstar(W)$ for $s \in S$.  In order to calculate the rank of $\rK_0(\Cstar(W))$, let us write $\rK_0(\Lambda) = \rK_0(\prod_\Lambda \CC^2)$ for a finite graph $\Lambda$.  Then $\Lambda \mapsto \rK_0(\Lambda)$ is a functor from the category of finite graphs, with morphisms given by induced inclusions, to the category of abelian torsion free groups.  It maps the empty graph to $\ZZ$ and satisfies the following rules:
  \begin{itemize}
  \item if $v \in \rV(\Lambda)$ is such that $\Star(v) = \Lambda$ then the following diagram commutes.
    \begin{gather*}
      \xymatrix{
        \rK_0(\Lambda \setminus v) \ar[dr]^{a \mapsto (1,-1) \otimes a} \ar[r] & \rK_0(\Lambda) \ar[d]^\cong \\
        & \bZ^2 \otimes \rK_0(\Lambda \setminus v)
      }
    \end{gather*}
  \item if $v \in \rV(\Lambda)$ satisfies $\Star(v) \neq \Lambda$ then we have a short exact sequence
    \begin{gather*}  \xymatrix{
        0 \ar[r] & \rK_0(\Link(v)) \ar[r] & \rK_0(\Star(v)) \oplus \rK_0(\Lambda \setminus v)  \ar[r] & \rK_0(\Lambda) \ar[r]& 0}
    \end{gather*}	
  \end{itemize}
  It follows that the rank $N(\Lambda)$ of $\rK_0(\Lambda)$ evaluated on the graph with one vertex is $N(K_1) = 2$.  Further,
  \begin{gather*}
    N(\Lambda) = N(\Link(v)) + N(\Lambda \setminus v)
  \end{gather*}
  for all $v \in \rV(\Lambda)$.  By Lemma \ref{lem:counting-cliques} and a straightforward induction we thus find that $N(\Lambda)$ is the number of all cliques in $\Lambda$, including the empty clique.  We next show that $(p_C)_{C \in \cC(\Lambda)}$ is a basis for $\rK_0(\Lambda)$, where $\cC(\Lambda)$ denotes the set of cliques of $\Lambda$.  Since $|\cC(\Lambda)|$ equals the rank of $\rK_0(\Lambda)$, it suffices to show that $(p_C)_{C \in \cC(\Lambda)}$ is generating.  This follows by induction, because choosing $v \in \Lambda$ such that $\Star(v) \neq \Lambda$, we saw that $\rK_0(\Lambda)$ is generated by $\rK_0(\Star(v)) \oplus \rK_0(\Lambda \setminus v)$.  Since every clique of $\Lambda$ is either contained in $\Star(v)$ or in $\Lambda \setminus v$, we can conclude the proof.
\end{proof}

\begin{remark}
As mentioned in the introduction, the abstract form of $\rK$-theory groups for (reduced) \Cstar-algebras of even, so in particular also right-angled Coxeter groups has been computed in \cite{sanchezgarcia07}. The proof there is based on first computing the Bredon homology of the Coxeter group in question, noting that it coincides with the equivariant $\rK$-homology, and then appealing to the main result of \cite{higsonkasparov97} which says that the Baum-Connes assembly map is an isomorphism for the groups which have the Haagerup property. The latter property holds for Coxeter groups due to \cite{bozejkojanuszkiewiczspatzier88}. The form of the $\rK_0$-group given by S\'anchez-Garcia in the right-angled case is expressed in terms of the number of spherical subsets of  $S$, which is easily seen to be equal to the number of cliques of the commutaion graph of $(W,S)$.
It should be noted that the techniques of \cite[Section 9]{sanchezgarcia07} in principle make it possible to compute, with the computer assistence,  the $\rK$-theory groups for $\Cstar$-algebras of arbitrary Coxeter groups. It is however highly non-clear how these could be applied to the associated Hecke \Cstar-algebra beyond the right-angled case. The methods used by Solleveld in \cite{solleveld18} in the affine case are of a completely different nature to these of our work, and use rather the explicit representation theory of the deformed algebras, as investigated in \cite{solleveld12}.	
\end{remark}

\begin{example}
  \label{ex:identical-k-theory}
  It is natural to ask whether $\rK$-theory of group \Cstar-algebras distinguishes the Coxeter type of a right-angled Coxeter group.  This is not the case, as demonstrated by the two examples associated with the following commutation graphs.
  \begin{center}
    \begin{minipage}{0.6\linewidth}
      \begin{tikzpicture}
        \node[inner sep=2pt,circle, fill] (A) at (0,0) {};
        \node[inner sep=2pt,circle, fill] (B) at (1,0) {};
        \node[inner sep=2pt,circle, fill] (C) at (2,0) {};
        \node[inner sep=2pt,circle, fill] (D) at (3,0) {};

        \draw (A) -- (B);
        \draw (B) -- (C);
      \end{tikzpicture}
      \hfill
      \begin{tikzpicture}
        \node[inner sep=2pt,circle, fill] (A) at (0,0) {};
        \node[inner sep=2pt,circle, fill] (B) at (1,0) {};
        \node[inner sep=2pt,circle, fill] (C) at (2,0) {};
        \node[inner sep=2pt,circle, fill] (D) at (3,0) {};

        \draw (A) -- (B);
        \draw (C) -- (D);
      \end{tikzpicture}
    \end{minipage}
  \end{center}
  The first group can be describes as the free product $((\ZZ/2\ZZ * \ZZ/2\ZZ) \oplus \ZZ/2\ZZ) * \ZZ/2\ZZ$, while the second group is isomorphic with $(\ZZ/2\ZZ)^{\oplus 2} * (\ZZ/2\ZZ)^{\oplus 2}$.  They have the same K-theory by Theorem \ref{thm:k-theory-computation}.  In Example \ref{ex:identical-trace-pairing}, we investigate the reduced group \Cstar-algebras of these groups, showing that adding other information from the Elliott invariant does not help to recover the Coxeter type.
\end{example}

The next result transfers the K-theory calculations obtained for universal group \Cstar-algebra in Theorem \ref{thm:k-theory-computation} to the setting of Hecke \Cstar-algebras.  Recall Notation~\ref{not:clique-projections} for projections associated with cliques.
\begin{corollary}
  \label{cor:k-theory-hecke-algebras}
  Let $(W,S)$ be a right-angled Coxeter system and $\qm \in (0,1]^S$.
  \begin{itemize}
  \item The map $C \mapsto [p_C]$ induces an isomorphism $\ZZ^\cC \cong \rK_0(\Cstar_{\qm}(W))$.
  \item $\rK_1(\Cstar_{\qm}(W)) = 0$.
  \end{itemize}
\end{corollary}
\begin{proof}
  This follows from Theorem \ref{thm:reduction-racg-kk-equivalence} and Theorem \ref{thm:k-theory-computation}.
\end{proof}

Now that we have found an explicit basis of $\rK_0(\Cstar_{\qm}(W))$, we can calculate the trace pairing with the natural trace.
\begin{proposition}
  \label{prop:trace-pairing}
  Let $(W, S)$ be a right-angled Coxeter system and $\qm \in (0,1]^S$.  Denote by $\cC$ the set of cliques of the commutation graph of $(W,S)$.  The trace pairing of $\rK_0(\Cstar_{\qm}(W))$ with the natural trace is given by
  \begin{gather*}
    \tau_*([p_C]) = \prod_{s \in C} \frac{1 }{1+q_s}
    \eqstop
  \end{gather*}
\end{proposition}
\begin{proof}
  This follows directly from the explicit description of the projection $p_C$ as a tensor product, and the computation in the last part of Section \ref{sec:coxeter-systems}.
\end{proof}

\begin{example}
  \label{ex:trace-pairing-isomorphic}
  Let $W_1, W_2$ be right-angled Coxeter groups with commutation graphs $\Gamma_1$ and $\Gamma_2$, respectively.  Then Proposition \ref{prop:trace-pairing} shows that $(\rK_0(\Cstar_r(W_1)), (\tau_1)_*) \cong (\rK_0(\Cstar_r(W_2)), (\tau_2)_*)$ if for every $n \in \NN$ the number of $n$-cliques in $\Gamma_1$ and $\Gamma_2$ agree.
\end{example}

\begin{example}
  \label{ex:identical-trace-pairing}
  Having pointed out two right-angled Coxeter groups with isomorphic K-groups in Example \ref{ex:identical-k-theory}, one could speculate whether adding other information from the Elliott invariant helps to remedy the situation.  That this is not the case, can be shown thanks to Proposition \ref{prop:trace-pairing}.  Considering the reduced group \Cstar-algebras of $W_1 = ((\ZZ/2\ZZ * \ZZ/2\ZZ) \oplus \ZZ/2\ZZ) * \ZZ/2\ZZ$ and $W_2 = (\ZZ/2\ZZ)^{\oplus 2} * (\ZZ/2\ZZ)^{\oplus 2}$, we find identifications of $(\rK_0(\Cstar_r(W_i), \tau_*)$ with $(\ZZ^6, (1, 1/2 , 1/2, 1/2, 1/2, 1/4, 1/4))$.   Since the reduced group \Cstar-algebras of $W_1$ and $W_2$ fall in the scope of the work of Dykema's result \cite[Theorem 1]{dykema99}, they have a unique trace.  Further, Dykema and R{\o}rdam's work \cite[Theorem 2]{dykemarordam1998}) shows that the order in $\rK_0$ is determined by this trace.  So the Elliott invariants of $\Cstarred(W_1)$ and $\Cstarred(W_2)$ are isomorphic, whereas the groups themselves are not isomorphic, because their commutation graphs are non-isomorphic.
\end{example}

Specialising Proposition \ref{prop:trace-pairing} to the case $\qm = {\bf 1}$, we obtain the following result which recovers the size of a maximal clique in the commutation graph.
\begin{corollary}
  \label{cor:trace-image-reduced}
  Let $(W,S)$ be a right-angled Coxeter group and $c$ the size of a maximal clique in the commutation graph of $(W,S)$.  Then $\tau_*(\rK_0(\Cstar_r(W))) = \frac{1}{2^c}\ZZ$.
\end{corollary}

\section{Classification of Hecke \Cstar-algebras by $\rK$-theory}
\label{sec:classification}

In this section we will focus on the case $\ZZ/2\ZZ^{*n}$ and understand by means of this group to which extent K-theoretic invariants can be used to classify Hecke \Cstar-algebras.  The next example collects the necessary K-theoretic calculations.
\begin{example}
  \label{ex:k-elliott-invariant-free-product}
  It follows from Dykema's work on simplicity of free products of finite dimensional abelian \Cstar-algebras and on traces on such free products \cite[Theorem 4.8]{dykema99} that for the single parameter case $\qm = (q, \dotsc, q) \in (0,1]^S$, the Hecke \Cstar-algebras $\Cstar_q(\ZZ/2\ZZ)$ is
  \begin{itemize}
  \item simple with a unique trace if and only if $q > \frac{1}{n-1}$,
  \item has a character whose kernel is a simple, non-unital \Cstar-algebra with a unique trace if $q = \frac{1}{n-1}$, and
  \item a direct sum of a one dimensional \Cstar-algebra with a simple \Cstar-algebra having a unique trace if $q < \frac{1}{n-1}$.
  \end{itemize}
  Let us consider the first and last case and calculate the (unordered) Elliott invariant of $\Cstar_q(\ZZ/2\ZZ^{*n})$.   Applying Corollary \ref{cor:k-theory-hecke-algebras}, we find that a basis of $\rK_0(\Cstar_q(\ZZ/2\ZZ^{*n})$ is given by $[1], [p_1], \dotsc, [p_n]$, where $p_i \in \Cstar_q(\ZZ/2\ZZ^{*n})$ denotes the spectral projection of the i-th copy of $\ZZ/2\ZZ$.  By Proposition \ref{prop:trace-pairing}, the trace pairing with the natural trace is then given under this identification by the row vector $(1 , \frac{1}{1 + q} , \dotsc , \frac{1}{1 + q})$.  If $q > \frac{1}{n-1}$, this describes the complete unordered Elliott invariant.  We thus assume from now on $q < \frac{1}{n-1}$ and will explicitly identify the two extremal traces on $\Cstar_q(\ZZ/2\ZZ^{*n})$ and their pairing with $\rK$-theory.

  It follows from Dykema's \cite[Theorem 4.8]{dykema99} that the one dimensional direct summand of $\Cstar_q(\ZZ/2\ZZ^{*n})$ is generated by the projection $p = \bigwedge_{1 \leq i \leq n} p_i$.  One extremal trace on $\Cstar_q(\ZZ/2\ZZ^{*n})$ is thus characterised by the formula
  \begin{gather*}
    \epsilon(x) p =  x p
    \eqstop
  \end{gather*}
  Under the identification with $\ZZ^{n+1}$ its pairing with $\rK$-theory is hence given by the constant row vector $(1 , \dotsc , 1)$.

  We will next identify the second extremal trace $\vphi$ on $\Cstar_q(\ZZ/2\ZZ^{*n})$.  It satisfies $\vphi(p) = 0$, so that the ansatz $\tau = t \epsilon + (1 - t)\vphi$ yields $t = \tau(p)$ and
  \begin{gather*}
    \vphi = \frac{\tau - t \epsilon}{1 - t}
    \eqstop
  \end{gather*}
  The calculation of $\tau(p)$ follows from the explicit identification of Hecke eigenvectors made for example in \cite[Proof of Theorem 5.3]{garncarek16} and \cite[Proposition 2.2]{raumskalski20}.  The vector
  \begin{gather*}
    \eta = \sum_{w \in \ZZ/2\ZZ^{*n}} q^{\frac{|w|}{2}} \delta_w \in \ltwo(\ZZ/2\ZZ^{*n})
  \end{gather*}
  lies in the image of $p$ and we have
  \begin{gather*}
    p \delta_e
    = \langle \delta_e, \frac{\eta}{\|\eta\|} \rangle  \frac{\eta}{\| \eta \|}
    = \frac{\eta_e}{\|\eta\|^2} \eta
    \eqstop
  \end{gather*}
  Hence
  \begin{gather*}
   t= \tau(p) = \frac{\eta_e^2}{\|\eta\|^2} = \left ( \sum_{w \in \ZZ/2\ZZ^{*n}} q^{|w|} \right )^{-1} = \frac{1 - (n-1)q}{1 + q}
    \eqstop
  \end{gather*}
  Here we used the calcuation of the growth series for $\ZZ/2\ZZ^{*n}$, which is based on the fact that the number $s_k$ of elements of length $k \in \NN$ is
  \begin{gather*}
    s_k =
    \begin{cases}
      1 & k = 0\ \\
      n (n-1)^{k-1} & k \neq 0
      \eqstop      
    \end{cases}
  \end{gather*}
  Indeed, it follows that
  \begin{gather*}
    \| \eta \|^2
    =
    \sum_{k \in \NN} s_k q^k
    =
    1 + n \sum_{k \in \NN} (n - 1)^k q^{k+1}
    =
    1 + \frac{nq}{1 - (n - 1)q}
    =
    \frac{1 + q}{1 - (n-1)q}
    \eqstop
  \end{gather*}
  For the spectral projection $p_i$ associated to the Coxeter generator, we find that
  \begin{align*}
    \vphi(p_i)
    & =
      \frac{\tau(p_i) - t \epsilon(p_i)}{1 - t} \\
    & =
      \frac{\frac{1}{1 + q} - \frac{1 - (n-1)q}{1 + q}}{1 - \frac{1 - (n-1)q}{1 + q}} \\
    & =
      \frac{(n-1)q}{1 + q} \cdot \frac{1 + q}{1 + q - 1 + (n-1)q} \\
    & =
      \frac{(n-1)q}{1 - (n-1)q} \\
    & =
      \frac{n-1}{n}
      \eqstop
  \end{align*}
  It follows that the K-theory pairing of $\vphi$ under the identification $\rK_0(\Cstar_q(\ZZ/2\ZZ)^{*n})$ is given by the row vector
  \begin{gather*}
    \left ( 1, \frac{n-1}{n}, \dotsc , \frac{n-1}{n} \right )
    \eqstop
  \end{gather*}
\end{example}

In order to exploit the explicit calculations made in Example \ref{ex:k-elliott-invariant-free-product}, we need two auxiliary lemmas.  The following elementary lemma is well-known.
\begin{lemma}\label{lem:Z+Zmu}
  Let $x,y \in [\frac{1}{2},1]$. Then  $\ZZ + x \ZZ  = \ZZ + y \ZZ$ if and only if either
  \begin{enumerate}
  \item $x = y$, or
  \item $x, y$ are rational and have the same order in $\QQ/\ZZ$.
  \end{enumerate}	
\end{lemma}
\begin{proof}
  If $\ZZ + x \ZZ  = \ZZ +  y \ZZ$ holds, then there exist integers $k, l, m, n \in \ZZ$ such that $x = k + ly$ and $y = m + nx$.  Hence $x = k + l(m + nx)$, so that $x(1-ln) \in \ZZ$ holds.  On the one hand, if $x$ is irrational, we must have $ln = 1$.  Since $x, y \in [\frac{1}{2}, 1]$, substituting $l = \pm 1$ in the equality $x = k + l y$ implies that $x = y$.  On the other hand, if $x$ is rational and $d \in \NN$ is the smallest non-zero positive integer such that $d x \in \NN$, then we have $\frac{1}{d} \in \ZZ + x \ZZ $, and obviously also $x \in \ZZ + \frac{1}{d}\ZZ$.  So $\ZZ + \frac{1}{d}\ZZ = \ZZ + x \ZZ$.  It then remains to note that if $c,d \in \NN_{\geq 1}$ then $\ZZ + \frac{1}{c} \ZZ = \ZZ + \frac{1}{d} \ZZ$ implies $c = d$.
\end{proof}

The next result is a consequence of the lemma above and the classification of orbits of the $\mathrm{GL}_n(\Z)$ action on $\RR^n$, as studied for example in \cite{cabrermundici17}. 
\begin{lemma}
  \label{lem:affine-orbits}
  Suppose that $x,y \in \RR$, let $n \in \NN_{\geq 3}$ and write $\mathbf{x} = (x, \dotsc, x) \in \RR^n$, $\mathbf{y} = (y, \dotsc, y) \in \RR^n$. Then the following conditions are equivalent:
  \begin{enumerate}
  \item \label{it:affine-transformation}
    there exist $B \in \mathrm{GL}_{n}(\ZZ)$ and  $C \in \ZZ^n$ such that 
    $\mathbf{x} = C + B \mathbf{y}$; 
  \item \label{it:equal-or-same-order}
    either $x=y$ or both $x$ and $y$ are rational and have the same order in $\QQ/\ZZ$.
  \end{enumerate}
\end{lemma}
\begin{proof}
  Following \cite{cabrermundici17} we define additive subgroups of $\RR$ by $G_x = \mathbb{Z}+ x\mathbb{Z}$ and $G_y = \mathbb{Z}+ y\mathbb{Z}$.  Note that the rank of $G_x$ is one, if $x \in \QQ$ and two if $x \notin \QQ$.  Since $n \geq 3$, we may apply \cite[Corollary 16]{cabrermundici17} and infer that \ref{it:affine-transformation} holds if and only if $G_x = G_y$.  The claim follows then from Lemma \ref{lem:Z+Zmu}.
\end{proof}

The next theorem demonstrates the extend to which K-theoretic invariants are able to classify Hecke \Cstar-algebras.  We comment on the use of the unordered Elliott invariant in contrast to the full Elliott invariant in Remark \ref{rem:unordered-elliott-invariant}. 
\begin{theorem}
  \label{thm:classification}
  Suppose $n \in \NN_{\geq 3}$ and let $q_1,q_2 \in (0,1]$.  Then the following statements hold true.
  \begin{itemize}
  \item If $\Cstar_{q_1}(\ZZ/2\ZZ^{*n}) \cong \Cstar_{q_2}(\ZZ/2\ZZ^{*n})$ and one of $q_1$ or $q_2$ lies in either of the sets $(0,\frac{1}{n-1})$, $\{\frac{1}{n-1}\}$, $(\frac{1}{n-1}\, 1]$, then so does the other.
  \item Assume $\Cstar_{q_1}(\ZZ/2\ZZ^{*n}) \cong \Cstar_{q_2}(\ZZ/2\ZZ^{*n})$ and $q_1, q_2 \in (\frac{1}{n-1}, 1]$.
    \begin{itemize}
    \item If $q_1$ or $q_2$ is irrational, then $q_1 = q_2$.
    \item If $q_1$ or $q_2$ is rational, then so is the other. Further, $\frac{1}{1 + q_1}$ and $\frac{1}{1 + q_2}$ have the same order in $\QQ/\ZZ$.
    \item If $q_1, q_2$ are rational and $\frac{1}{1 + q_1}$, $\frac{1}{1 + q_2}$ have the same order in $\QQ/\ZZ$, then the unordered Elliott invariants of $\Cstar_{q_1}(\ZZ/2\ZZ^{*n})$ and $\Cstar_{q_2}(\ZZ/2\ZZ^{*n})$ are isomorphic.
    \end{itemize}
  \item If $q_1, q_2 \in (0,\frac{1}{n-1})$, then the unordered Elliott invariants of $\Cstar_{q_1}(\ZZ/2\ZZ^{*n})$ and $\Cstar_{q_2}(\ZZ/2\ZZ^{*n})$ are isomorphic.
  \end{itemize}
\end{theorem}
\begin{proof}
  The division into the three different sets follows from Dykema's results established in \cite{dykema99}.  We first observe that if $q_1,q_2 \in (0,\frac{1}{n-1})$ then Example \ref{ex:k-elliott-invariant-free-product} shows that the unordered Elliott invariant of both Hecke \Cstar-algebras is up to isomorphism described by
  \begin{itemize}
  \item $\rK_0(\Cstar_{q_i}(\ZZ/2\ZZ^{*n}) \cong \ZZ^{n+1}$ with order unit $(1, 0, \dotsc, 0)^\rmt$,
  \item $\rK_1(\Cstar_{q_i}(\ZZ/2\ZZ^{*n}) \cong 0$,
  \item $\rT(\Cstar_{q_i}(\ZZ/2\ZZ^{*n})) \cong [0,1]$,
  \item the pairing of the trace $\epsilon$, corresponding to $0 \in [0,1]$, is given by the row vector $(1, \dotsc, 1)$, and
  \item the pairing of the trace $\vphi$, corresponding to $1 \in [0,1]$, is given by the row vector $(1, \frac{n-1}{n}, \dotsc, \frac{n-1}{n})$.
  \end{itemize}

  So it remains to investigate the case $q_1, q_2 \in (\frac{1}{n-1}, 1]$.  The unordered Elliott invariant of $\Cstar_{q_i}(\ZZ/2\ZZ^{*n})$ is calculated in Example \ref{ex:k-elliott-invariant-free-product}.
  \begin{itemize}
  \item $\rK_0(\Cstar_{q_i}(\ZZ/2\ZZ^{*n}) \cong \ZZ^{n+1}$ with order unit $(1, 0, \dotsc, 0)^\rmt$,
  \item $\rK_1(\Cstar_{q_i}(\ZZ/2\ZZ^{*n}) \cong 0$,
  \item $\rT(\Cstar_{q_i}(\ZZ/2\ZZ^{*n})) \cong \{0\}$, and
  \item the pairing of the unique trace $\tau$ is given by the row vector $(1 , \frac{1}{1+q_i}, \dotsc, \frac{1}{1+q_i})$.
  \end{itemize}

  Automorphisms of $\ZZ^{n+1}$ are given by left multiplication with $A \in \mathrm{GL}_{n+1}$.  Such automorphism preserves the unit of $\rK$-theory $(1,0, \dotsc, 0)^\rmt$ if and only if
  \begin{gather*}
    A =
    \begin{pmatrix}
      1 & C \\
      0 & B
    \end{pmatrix}
  \end{gather*}
  with $B \in \mathrm{GL}_n(\ZZ)$ and  $C \in \ZZ^{1 \times n}$.  Further, the pairing with the unique trace is preserved if and only if
  \begin{gather*}
    (1,y)
    =
    \left (1 , \frac{1}{1+q_2}, \dotsc, \frac{1}{1+q_2} \right )
    =
    \left (1 , \frac{1}{1+q_1}, \dotsc, \frac{1}{1+q_1} \right ) A
    =
    \left (1 , xB + C \right )
    \eqcomma
  \end{gather*}
  where $y = (\frac{1}{1 + q_2}, \dotsc,  \frac{1}{1 + q_2})$ and $x = (\frac{1}{1 + q_1}, \dotsc, \frac{1}{1 + q_1})$.  So Lemma \ref{lem:affine-orbits} applies and says that the unordered Elliott invariants of $\Cstar_{q_1}(\ZZ/2\ZZ^{*n})$ and $\Cstar_{q_2}(\ZZ/2\ZZ^{*n})$ are isomorphic if and only if one of the following two statements hold:  either $\frac{1}{1 + q_1}$ and $\frac{1}{1 + q_2}$ are irrational and equal to each other;  or $\frac{1}{1 + q_1}$ and $\frac{1}{1 + q_2}$ are rational and have the same order in $\QQ/\ZZ$.   The first case accounts exactly for $q_1$ and $q_2$ being irrational.  The second case accounts exactly for $q_1$ and $q_2$ being rational. 
\end{proof}

\begin{remark}
  \label{rem:unordered-elliott-invariant}
  In Theorem \ref{thm:classification}, we only make use of the unordered Elliott invariant.  In contrast, the classification programme for amenable \Cstar-algebras considers the order structure on the $\rK_0$-group as part of their invariant.  In practice in all situations where the order structure could be calculated, it turned out to be actually described by the trace pairing of the Elliott invariant.  Calculating the order structure of $\rK$-theory for free products, extending work of Dykema-R{\o}rdam \cite{dykemarordam1998} beyond the so called Avitzour condition, seems to be a challenging task that deserves separate attention.
\end{remark}

\begin{remark}
  \label{rem:distinguishing-tempered-representations}
  Despite the limits of classification results for Hecke \Cstar-algebras with rational parameters described by Theorem \ref{thm:classification}, one can make the following observation.  Given $d \in \NN_{\geq 3}$, the Hecke \Cstar-algebra $\Cstar_q(\ZZ/2\ZZ^{*n})$ with parameter $q = \frac{1}{d}$ does recognise $d$ as long as $d \leq n - 1$.  Indeed,
  \begin{gather*}
    \frac{1}{1 + q} = \frac{1}{\frac{1 + d}{d}} = \frac{d}{d+1}
  \end{gather*}
has order $d + 1$ in $\QQ/\ZZ$.  This observation applies to the Iwahori-Hecke \Cstar-algebra associated with groups acting strongly transitively on regular buildings of type $\ZZ/2\ZZ^{*n}$ and thickness $d \in \{2, \dots, n-1\}$.  See \cite[Section 2]{raumskalski20} for details on the relation between Hecke \Cstar-algebras and Iwahori-Hecke \Cstar-algebras.
\end{remark}


\bibliographystyle{plain}
\bibliography{k-theory-hecke-cstar-algebras-revised}

\begin{thebibliography}{10}

\bibitem{blackadar98}
Bruce Blackadar.
\newblock {\em K-theory for operator algebras}, volume~5 of {\em Mathematical
  Sciences Research Institute Publications}.
\newblock Cambridge University Press, Cambridge, second edition, 1998.

\bibitem{bocaradulescu92}
Florin Boca and Florin R{\u a}dulescu.
\newblock {Singularity of radial subalgebras in II$_1$ factors associated with
  free products of groups.}
\newblock {\em J. Funct. Anal.}, 103(1):138--159, 1992.

\bibitem{bozejkojanuszkiewiczspatzier88}
M.~Bo\.{z}ejko, T.~Januszkiewicz, and R.~J. Spatzier.
\newblock Infinite {C}oxeter groups do not have {K}azhdan's property.
\newblock {\em J. Operator Theory}, 19(1):63--67, 1988.

\bibitem{breuillardkalantarkennedyozawa14}
Emmanuel Breuillard, Mehrdad Kalantar, Matthew Kennedy, and Narutaka Ozawa.
\newblock {C$^*$-simplicity and the unique trace property for discrete groups.}
\newblock {\em Publ. Math. Inst. Hautes \'Etud. Sci.}, 126(1):35--71, 2017.

\bibitem{cabrermundici17}
Leonardo~Manuel Cabrer and Daniele Mundici.
\newblock Classifying orbits of the affine group over the integers.
\newblock {\em Ergodic Theory Dyn. Syst.}, 37(2):440--453, 2017.

\bibitem{capracelecureux11}
Pierre-Emmanuel Caprace and Jean L{\'e}cureux.
\newblock {Combinatorial and group-theoretic compactifications of buildings}.
\newblock {\em Ann. Inst. Fourier}, 61(2):619--672, 2011.

\bibitem{cgstw19}
Jos{\'e} Carri{\'o}n, James Gabe, Christopher Schafhauser, Aaron Tikuisis, and
  Stuart White.
\newblock {Classification of $^*$-homomorphims I: simple nuclear
  C$^*$-algebras.}
\newblock In preparation, 2021.

\bibitem{caspers20}
Martijn Caspers.
\newblock {Absence of Cartan subalgebras for right-angled Hecke von Neumann
  algebras.}
\newblock {\em Anal. PDE}, 13(1):1--28, 2020.

\bibitem{caspersfima17}
Martijn Caspers and Pierre Fima.
\newblock {Graph products of operator algebras.}
\newblock {\em J. Noncommut. Geom.}, 11(1):367--411, 2017.

\bibitem{caspersklisselarsen19}
Martijn Caspers, Mario Klisse, and Nadia Larsen.
\newblock {Graph product Khintchine inequalities and Hecke C$^*$-algebras:
  Haagerup inequalities.}, (non)simplicity, nuclearity and exactness.
\newblock {\em J. Funct. Anal.}, 280(108795), 2021.

\bibitem{capsersskalskiwasilewski19}
Martijn Caspers, Adam Skalski, and Mateusz Wasilewski.
\newblock {On masas in q-deformed von Neumann algebras.}
\newblock {\em Pacific J. Math.}, 302(1):1--21, 2019.

\bibitem{chifansinclair13}
Ionuts Chifan and Thomas Sinclair.
\newblock {On the structural theory of II$_1$ factors of negatively curved
  groups.}
\newblock {\em Ann. Sci. {\'E}c. Norm. Sup{\'e}r. (4)}, 46(1):1--34, 2013.

\bibitem{cuntz82-internal-structure}
Joachim Cuntz.
\newblock {The internal structure of simple C$^*$-algebras.}
\newblock In {\em Operator algebras and applications, Part I (Kingston, Ont.,
  1980)}, volume~38 of {\em Proceedings of Symposia in Pure Mathematics}, pages
  85--115. Providence, RI: American Mathematical Society, 1982.

\bibitem{cuntz82-free-products}
Joachim Cuntz.
\newblock {The K-groups for free products of C$^*$-algebras.}
\newblock In {\em Operator algebras and applications, Part I (Kingston, Ont.,
  1980)}, volume~38 of {\em Proceedings of Symposia in Pure Mathematics}, pages
  81--84. Providence, RI: American Mathematical Society, 1982.

\bibitem{cuntz83-k-amenability}
Joachim Cuntz.
\newblock {$K$-theoretic amenability for discrete groups.}
\newblock {\em J. Reine Angew. Math.}, 344:180--195, 1983.

\bibitem{davis08}
Michael~W. Davis.
\newblock {\em {The geometry and topology of Coxeter groups}}, volume~32 of
  {\em London Mathematical Society Monographs Series}.
\newblock Princeton, NJ: Princeton University Press, 2008.

\bibitem{davisdymarajanusykiewiczokun07}
Michael~W. Davis, Jan Dymara, Tadeusy Januszkiewicz, and Boris Okun.
\newblock {Weighted L$^2$-cohomology of Coxeter groups.}
\newblock {\em Geom. Topol.}, 11:47--138, 2007.

\bibitem{dykema93-hyperfinite}
Kenneth~J. Dykema.
\newblock {Free products of hyperfinite von Neumann algebras and free
  dimension.}
\newblock {\em Duke Math. J.}, 69(1):97--119, 1993.

\bibitem{dykema99}
Kenneth~J. Dykema.
\newblock {Simplicity and the stable rank of some free product C$^*$-algebras.}
\newblock {\em Trans. Am. Math. Soc.}, 351(1):1--40, 1999.

\bibitem{dykemarordam1998}
Kenneth~J. Dykema and Mikael R{\o}rdam.
\newblock {Projections in free product C$^*$-algebras.}
\newblock {\em Geom. Funct. Anal.}, 8(1):1--16, 1998.

\bibitem{eckhardtraum18}
Caleb Eckhardt and Sven Raum.
\newblock {C$^*$-superrigidity of 2-step nilpotent groups.}
\newblock {\em Adv. Math.}, 338:175--195, 2018.

\bibitem{eilersliruiz16}
S{\o}ren Eilers, Xin Li, and Efren Ruiz.
\newblock {The isomorphism problem for semigroup C$^*$-algebras of right-angled
  Artin monoids.}
\newblock {\em Doc. Math.}, 21:309--343, 2016.

\bibitem{eliaswilliamson14}
Ben Elias and Geordie Williamson.
\newblock {The Hodge theory of Soergel bimodules.}
\newblock {\em Ann. Math. (2)}, 180(3):1089--1136, 2014.

\bibitem{elliott76}
George~A. Elliott.
\newblock On the classification of inductive limits of sequences of semisimple
  finite-dimensional algebras.
\newblock {\em J. Algebra}, 38(1):29--44, 1976.

\bibitem{fimagermain15}
Pierre Fima and Emmanuel Germain.
\newblock {The $KK$-theory of amalgamated free products.}
\newblock {\em Adv. Math.}, 369(107174):1--35, 2020.

\bibitem{garncarek16}
{\L}ukasz Garncarek.
\newblock {Factoriality of Hecke–von Neumann algebras of right-angled Coxeter
  groups.}
\newblock {\em J. Funct. Anal.}, 270(3):1202--1219, 2016.

\bibitem{germain96}
Emmanuel Germain.
\newblock {$KK$}-theory of reduced free-product {$C^*$}-algebras.
\newblock {\em Duke Math. J.}, 82(3):707--723, 1996.

\bibitem{germain97}
Emmanuel Germain.
\newblock {{$KK$}-theory of the full free product of unital C$^*$-algebras}.
\newblock {\em J. Reine Angew. Math.}, 485:1--10, 1997.

\bibitem{gonglinniu15}
Guihua Gong, Huaxin Lin, and Zhuang Niu.
\newblock {Classification of finite simple amenable $\mathcal Z$-stable
  C$^*$-algebras.}
\newblock Preprint, 2015.

\bibitem{green-thesis}
Elisabeth~R. Green.
\newblock {\em {Graph products.}}
\newblock PhD thesis, University of Leeds,
  {http://ethesis.whiterose.ac.uk/236}, 1990.

\bibitem{higsonkasparov97}
Nigel Higson and Gennadi~G. Kasparov.
\newblock Operator {$K$}-theory for groups which act properly and isometrically
  on {H}ilbert space.
\newblock {\em Electron. Res. Announc. Amer. Math. Soc.}, 3:131--142, 1997.

\bibitem{humphreys90}
James~E. Humphreys.
\newblock {\em {Reflection groups and Coxeter groups}}, volume~29 of {\em
  Cambridge Studies in Advanced Mathematics}.
\newblock Cambridge: Cambridge University Press, 1990.

\bibitem{kalantarkennedy14-boundaries}
Mehrdad Kalantar and Matthew Kennedy.
\newblock {Boundaries of reduced $\mathrm{C}^*$-algebras of discrete groups.}
\newblock {\em J. Reine Angew. Math.}, 727:247--267, 2017.

\bibitem{kazhdanlusztig79}
David Kazhdan and George Lusztig.
\newblock {Representations of Coxeter groups and Hecke algebras.}
\newblock {\em Invent. Math.}, 53(2):165--184, 1979.

\bibitem{klisse2020}
Mario Klisse.
\newblock {Topological boundaries of connected graphs and Coxeter groups.}
\newblock To appear in J.Op.Theory., 2020.

\bibitem{klisse2021}
Mario Klisse.
\newblock {Simplicity of right-angled Hecke C$^*$-algebras.}
\newblock To appear in Int. Math. Res. Not., 2021.

\bibitem{lamthomas15}
Thomas Lam and Anne Thomas.
\newblock {Infinite reduced words and the Tits boundary of a Coxeter group.}
\newblock {\em Int. Math. Res. Not.}, 2015(17):7690--7733, 2015.

\bibitem{meyernest06}
Ralf Meyer and Richard Nest.
\newblock {The Baum-Connes conjecture via localisation of categories}.
\newblock {\em Topology}, 45:209--259, 2006.

\bibitem{neeman01}
Amnon Neeman.
\newblock {\em {Triangulated categories}}, volume 148 of {\em Annals of
  Mathematics Studies}.
\newblock Princeton, Oxford: Princeton University Press, 2001.

\bibitem{ozawa04-solid}
Narutaka Ozawa.
\newblock {Solid von Neumann algebras}.
\newblock {\em Acta Math.}, 192(1):111--117, 2004.

\bibitem{pimsnervoiculescu82}
Mihai Pimsner and Dan~V. Voiculescu.
\newblock {K-groups of reduced crossed products by free groups.}
\newblock {\em J. Oper. Theory}, 8(1):131--156, 1982.

\bibitem{powers75}
Robert~T. Powers.
\newblock {Simplicity of the C$^*$-algebra associated with the free group on
  two generators.}
\newblock {\em Duke Math. J.}, 42:151--156, 1975.

\bibitem{radcliffe03}
David~G. Radcliffe.
\newblock Rigidity of graph products of groups.
\newblock {\em Algebr. Geom. Topol.}, 3:1079--1088, 2003.

\bibitem{raum19-bourbaki}
Sven Raum.
\newblock {C$^*$-simplicity (after Breuillard, Haagerup, Kalantar, Kennedy and
  Ozawa).}
\newblock {\em Ast{\'e}risque}, 422, 2020.
\newblock S{\'e}minaire Bourbaki, expos{\'e} 1158.

\bibitem{raumskalski20}
Sven Raum and Adam Skalski.
\newblock {Factorial multiparameter Hecke von Neumann algebras and
  representations of groups acting on right-angled buildings}.
\newblock Preprint, 2020.

\bibitem{sanchezgarcia07}
Rub\'{e}n~J. S\'{a}nchez-Garc\'{\i}a.
\newblock Equivariant {$K$}-homology for some {C}oxeter groups.
\newblock {\em J. Lond. Math. Soc. (2)}, 75(3):773--790, 2007.

\bibitem{skandalis88}
Georges Skandalis.
\newblock Une notion de nucl\'{e}arit\'{e} en k-th\'{e}orie (d'apr\`es {J}.
  {C}untz).
\newblock {\em $K$-Theory}, 1(6):549--573, 1988.

\bibitem{solleveld12}
Maarten Solleveld.
\newblock On the classification of irreducible representations of affine
  {H}ecke algebras with unequal parameters.
\newblock {\em Represent. Theory}, 16:1--87, 2012.

\bibitem{solleveld18}
Maarten Solleveld.
\newblock {Topological K-theory of affine Hecke algebras.}
\newblock {\em Ann. K-theory}, 3(3):395--460, 2018.

\bibitem{tikuisiswhitewinter17}
Aaron Tikuisis, Stuart White, and Wilhelm Winter.
\newblock {Quasidiagonality of nuclear C$^*$-algebras.}
\newblock {\em Ann. of Math. (2)}, 185:229--284, 2017.

\bibitem{tu99}
Jean-Louis Tu.
\newblock La conjecture de {B}aum-{C}onnes pour les feuilletages moyennables.
\newblock {\em $K$-Theory}, 17(3):215--264, 1999.

\bibitem{voiculescu85}
Dan~V. Voiculescu.
\newblock {Symmetries of some reduced free product C$^*$-algebras.}
\newblock In {\em Operator algebras and their connections with topology and
  ergodic theory, (Buşteni/Rom. 1983)}, volume 1132 of {\em Lect. Notes
  Math.}, pages 556--588. Berlin Heidelberg New York: Springer, 1985.

\end{thebibliography}


\vspace{2em}
{\small
  \parbox[t]{200pt}
  {
    Sven Raum \\
    Department of Mathematics \\
    Stockholm University \\
    SE-106 91 Stockholm \\
    Sweden \\
    {\footnotesize raum@math.su.se}
  }
  \parbox[t]{200pt}
  {
    Adam Skalski \\
    Institute of Mathematics of the \\ Polish Academy of Sciences \\
    ul.\ \'Sniadeckich 8 \\
    00-656 Warszawa \\
    Poland \\
    {\footnotesize a.skalski@impan.pl}
  }
}

\end{document}